\documentclass[amssymb,10pt]{article}
\usepackage[dvipdfm]{graphicx}
\usepackage{amsthm}
\usepackage[usenames]{color}
\usepackage{xspace,colortbl}
\usepackage{latexsym,url,amscd,amssymb}
\usepackage{amsmath}
\usepackage{hyperref}
\usepackage{cases}
\usepackage{caption}
\usepackage{algorithm}
\usepackage{algpseudocode}
\usepackage{listings}
\usepackage{rotating}
\usepackage{framed}
\usepackage{amssymb}
\usepackage[numbers,sort,compress]{natbib}%多篇文献引用格式
\usepackage{amsmath}\numberwithin{equation}{section}

\def\A{{\mathcal A}}
\def\B{{\mathcal B}}
\def\C{{\mathcal C}}

\def\G{{\mathcal G}}

\def\I{{\mathcal I}}

\def\L{{\mathcal L}}
\def\M{{\mathcal M}}
\def\N{{\mathcal N}}
\def\O{{\mathcal O}}

\def\R{{\mathcal R}}
\def\S{{\mathcal S}}
\def\T{{\mathcal T}}

\def\V{{\mathcal V}}

\def\X{{\mathcal X}}
\def\Y{{\mathcal Y}}
\def\Z{{\mathcal Z}}

\newtheorem{lemma}{Lemma}[section]
\newtheorem{definition}{Definition}[section]

\newtheorem{theorem}{Theorem}[section]
\newtheorem{assumption}{Assumption}[section]
\newtheorem{corollary}{Corollary}
\newtheorem{proposition}{Proposition}[section]
\renewcommand{\thealgorithm}{\arabic{section}.\arabic{algorithm}}

\title{Linear Convergence Rate Analysis of Proximal Generalized ADMM for Convex Composite Programming}
\author{
	Han Wang\thanks{School of Mathematics and Statistics, 	Henan University, Kaifeng 475000, China (Email: whmath@126.com).}
	\and
	Yunhai Xiao\thanks{Center for Applied Mathematics of Henan Province,
		Henan University, Zhengzhou 450046,  China (Email: yhxiao@henu.edu.cn).} \href{https://orcid.org/0000-0002-7503-4585}{\includegraphics[scale=0.08]{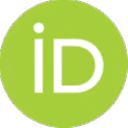}}
}
\date{\today}
\begin{document}
	\maketitle
\begin{abstract}
The proximal generalized alternating direction method of multipliers (p-GADMM)  is substantially efficient for solving convex composite programming problems of high-dimensional to moderate accuracy. 
The global convergence of this method was established by Xiao, Chen \& Li [{\tt Math. Program. Comput., 2018}], but its convergence rate was not given. 
One may take it for granted that the convergence rate could be proved easily by mimicking the proximal ADMM, but we find the relaxed points will certainly cause many difficulties for theoretical analysis.
In this paper, we devote to exploring its convergence behavior and show that the sequence generated by p-GADMM possesses Q-linear convergence rate under some mild conditions. We would like to note that the proximal terms at the subproblems are required to be positive definite, which is very common in most practical implementations although it seems to be a bit strong.
	
\end{abstract}
\section{Introduction}
%%%%%%%%%%%%%%%%%%%%%%%%%%%%%%%%%%%%%%%%%%%%
Let $\Y:=\Y_1\times \cdots\times\Y_m$ and $\Z:=\Z_1\times \cdots\times\Z_n$ be finite-dimensional real Euclidean spaces, each endowed with an inner product $ \langle \cdot,\cdot \rangle$ as well as its induced norm $\|\cdot \|$.
In this paper, we consider a canonical class of convex composite minimization problem with a separable objective function and linear constraint
\begin{equation}\label{eq:1}
	\begin{aligned}
		\min&{\;\:f(y)+g(z)}\\
		\textrm{s.t.}\:&\;\: \A^*y+\B^*z=c,
	\end{aligned}
\end{equation}
where $f:\Y\rightarrow \left ( -\infty ,+\infty  \right ]$ and $g:\Z\rightarrow \left ( -\infty ,+\infty  \right ]$ are closed proper convex but not necessarily smooth functions,  $\A:\X\rightarrow \Y$  and $\B:\X\rightarrow \Z$ are linear operators with their adjoints $\A^*$ and $\B^*$, respectively, $c\in \X$ is the given date.
Given this structure,  problem (\ref{eq:1}) arises in a
 wide variety of statistical
and machine learning problems of recent interest, including the lasso,
sparse logistic regression, basis pursuit, covariance selection, support
vector machines, and many others \cite{Boyd2011DistributedOA}.

To solve (\ref{eq:1}), a simple but powerful algorithm is the alternating direction method of multipliers (ADMM) designed originally by Glowinski \& Marroco \cite{glowinski1975approximation} and Gabay \& Mercier \cite{Gabay1976ADA}, whose construction was closely correlated with Rockafellar's work on proximal point algorithm (PPA) for solving a more general maximal monotone inclusion problem \cite{Rockafellar1976AugmentedLA,Rockafellar1976MonotoneOA}. The readers may refer to  \cite{Glowinski2014OnAD,Boyd2011DistributedOA} for  reviewing the historical development of ADMM. Starting from $(y^0,z^0,x^0)\in\Y\times\Z\times\X$, the iterative scheme of ADMM for solving (\ref{eq:1}) takes the following form, for $k=0,1,\ldots,$
\begin{equation}\label{eq:17}
	\left\{  
	\begin{array}{l}  
		y^{k+1} =   \underset{y\in\Y}{\arg\min}\;\L_\sigma (y,z^k;x^k),\\ 
		z^{k+1}  =  \underset{z\in\Z}{\arg\min}\;\L_\sigma (y^{k+1},z;x^k),\\ 
		x^{k+1}  = x^k-\tau \sigma (\A^*y^{k+1}+\B^*z^{k+1}-c),
	\end{array}  
	\right.  
\end{equation}
where $\tau\in(0,(1+\sqrt{5})/2)$ is a step-length, and $\L_\sigma (y,z;x)$ is the augmented Lagrangian function of problem (\ref{eq:1}) defined as
\begin{equation*}
	\L_\sigma (y,z;x):=f(y)+g(z)-\left\langle x,\A^*y+\B^*z-c\right\rangle +\frac{\sigma }{2}\left \| \A^*y+\B^*z-c \right \|^2\!,
\end{equation*}
where $x\in\X$ is a multiplier and $\sigma>0$ is a penalty parameter.
It is interesting to note that the iteration scheme (\ref{eq:17}) is not always well-defined,  one can  consult \cite{chen2017note} for a counter-example.

Under the existence condition of a solution to the Karush-Kuhn-Tucker (KKT) system of (\ref{eq:1}), Gabay \cite{gabay1983chapter} showed that the ADMM (\ref{eq:17}) with an unit steplength is actually equivalent to the well-known Douglas-Rachford splitting (DRs) method  \cite{Lions1979SplittingAF} to find a zero point to the stationary system coming from the dual of (\ref{eq:1}). Also,  DRs can be considered as an application of PPA, see  \cite{Eckstein1992OnTD,cai2022developments}. As a result,  Eckstein \& Bertsekas \cite{Eckstein1992OnTD} applied an accelerated technique to  (\ref{eq:17}) so as to getting a generalized variant ADMM, that is
\begin{equation}\label{eq:18}
	\left\{  
	\begin{array}{l}  
		y^{k+1} = \underset{y\in\Y}{\arg\min}\;\L_\sigma (y,z^k;x^k),\\ 
		z^{k+1}= \underset{z\in\Z}{\arg\min}\;g(z)-\left \langle \B x^k,z\right \rangle+\frac{\sigma }{2}\left \| \rho (\A^*y^{k+1}+\B^*z^k-c )+\B^*(z-z^k)\right \|^2,\\ 
		x^{k+1}=x^k-\sigma \Big(\rho (\A^*y^{k+1}+\B^*z^{k}-c)+\B^*(z^{k+1}-z^k)  \Big),
	\end{array}  
	\right.  
\end{equation}
where $\rho\in\left( 0,2\right) $ is a relaxation factor. It is quite clear to see that the  (\ref{eq:17}) with $\tau=1$ is consistent with  (\ref{eq:18}) under the setting $\rho=1$. 
The generalized ADMM retains the benefits of treating the objective functions $f$ and $g$ individually, and at the same time, it also enjoys the easiness to implement. Most importantly, a suitable $\rho$ may lead to better numerical performance. 
For the empirical studies of  the generalized ADMM, one can refer to \cite{bertsekas1982constrained,Cai2013APP,Eckstein1994ParallelAD,xiao2018generalized}. 

To make the subproblems in (\ref{eq:17}) admit unique solutions  without further assumptions on the objective functions and constraints, Eckstien \cite{Eckstein1994SomeSS}  suggested add a proximal term  to each subproblem. Later, Fazel et al. \cite{Fazel2013HankelMR} showed that these added proximal terms are not necessarily positive definite, and then proposed 
a more powerful but convenient semi-proximal ADMM (abbr. sPADMM).
The sPADMM covers the classic ADMM as a special case and has the ability to deal with multi-block convex composite semidefinite programming problems of a low to moderate accuracy \cite{li2015two,chen2017efficient,sun2015convergent}. In recent years, Xiao, Chen \& Li \cite{xiao2018generalized} introduced a variant of generalized ADMM with semi-proximal terms (p-GADMM), that is, starting from $\tilde{\omega}^0:=(\tilde{x}^0,\tilde{y}^0,\tilde{z}^0)\in\X\times\textrm{dom}f\times\textrm{dom}g$, it generates a sequence $\{(y^k,z^k,x^k)\}$ using the following frameworks
\begin{equation}\label{eq:19}
	\begin{cases}
		y^k:=\underset{y\in\Y}{\arg\min}\L_\sigma (y,\tilde{z}^{k};\tilde{x}^{k})+\frac{1}{2}\left \| y-\tilde{y}^k \right \|^2_\S,\\ 
		x^k:=\tilde{x}^k-\sigma (\A^*y^k+\B^*\tilde{z}^{k}-c),\\
		z^k:=\underset{z\in\Z}{\arg\min}\L_\sigma (y^{k},z;x^k)+\frac{1}{2}\left \| z-\tilde{z}^k \right \|^2_\T,\\
		\tilde{\omega}^{k+1}:=\tilde{\omega}^{k}+\rho (\omega^{k}-\tilde{\omega}^{k}).
	\end{cases}
\end{equation}
In fact, this variant is based on an important observation of Chen \cite{chen2012numerical} that the generalized ADMM (\ref{eq:18}) can be reformulated as an ADMM with an extra relaxation step with  factor lying in  $(0,2)$.  
By comparison with  \cite{lu2018convergence}, the semi-proximal terms  in (\ref{eq:19}) is more natural and pretty in sense that it used the most recent values of variables.  
Moreover, extensive numerical experiments on a class of linearly doubly non-negative semidefinite programming problems illustrated that the variant of generalized ADMM (\ref{eq:19}) performed more effectively and  efficiently \cite{xiao2018generalized}. 

It has been proved that the sequence generated by p-GADMM converges globally to the KKT point of (\ref{eq:1}) under some mild conditions. However, its convergence rate  was not given. 
In fact, the convergence rate analysis on ADMM and its related variants have been studied by  in different contexts. For example, under only an error bound condition, Han, Sun \& Zhang \cite{Han2018LinearRC} established the linear rate convergence rate of sPADMM of Fazel et al. \cite{Fazel2013HankelMR} with $\tau\in\left ( 0,\left ( 1+\sqrt{5} \right )/2 \right )$.
For another example, Fang et al. \cite{fang2015generalized} derived the linear convergence rate of a linearized variant of generalized ADMM and proved the worst-case $\O\left( 1/k\right)$ iteration complexity in both ergodic and nonergotic cases. 
This result was further improved by Wang et al. \cite{wang2022linear}, in which they showed that the Q-linear convergence result for generalized ADMM (\ref{eq:18}) hold if the proximal terms are positive and semidefinite.
But despite these achievements, the convergence rate of p-GADMM with respect to (\ref{eq:19}) is not trivial because the relaxed iterative points $(\tilde{x}^k,\tilde{y}^k,\tilde{z}^k)$ would certainly cause many difficulties.
The paper concentrates on  theoretical analysis to prove that the sequence $\{(y^k,z^k)\}$ generated by p-GADMM possesses Q-linear convergence rate under the condition of calmness.

The remaining parts of this paper are organized as follows. Section \ref{sec2} is divided into two subsections, Subsection \ref{sec2.1} presents some  results on the optimimality conditions for problem (\ref{eq:1}), and Subsection \ref{subsec2.2} briefly overviews the definitions and properties associated with calmness in variational analysis.  Section \ref{sec3} is the main part of this paper, in which,  we derive the local linear convergence results for p-GADMM under some certain assumption conditions. Finally, we conclude the paper in Section \ref{sec4}.

\section{Preliminaries}\label{sec2}
The section presents some notations and basic concepts appeared in the context, and summarizes some useful  preliminaries used for later analysis.
\subsection{Notations and basic concepts}\label{sec2.1}
For any two vectors $x\in \mathbb{R}^n$ and $y\in\mathbb{R}^m$, we use $(x,y)$ to denote their adjunction, i.e., $(x,y)\in\mathbb{R}^{m+n}$. For $p\geq1$, we use $\left \| x \right \|_p$ to denote an $\ell_p$-norm of a vector $x$, and for $p=2$, we simply denote it as $ \| x \|$. For any  symmetric and positive definite matrix $\O$, the $\O$-norm of $x$ is denoted by $\|x\|_\O:=\sqrt{x^\top\O x}$. For a given closed convex set $\C$, the distance of $x$ to $\C$ regarding $\O$-norm is denoted as $\text{dist}_\O(x,\C):=\textrm{inf}_{y\in \C}\| x-y\|_\O$. Give $f: \Y \rightarrow(-\infty,+\infty]$ be a proper closed convex function.
We use $\text{dom}f$ to denote the domain of $f$, that is, $\text{dom}f=\{y \in \Y\mid f(y)<\infty\}$. The proximal mapping of $f$ with  $t >0$ is defined by
$$
\operatorname{Prox}_{f}(x) :=\underset{y \in \Y}{\operatorname{argmin}}\Big\{f(y)+\frac{1}{2}\|y-x\|_2^{2}\Big\}, \quad \forall x \in \Y.
$$

The Lagrangian function of problem (\ref{eq:1}) is defined by 
\begin{equation}\label{eq:2}
	\L(y,z;x):=f(y)+g(z)-\left \langle x,\A^*y+\B^*z-c \right \rangle,\; \; \; \forall (y,z,x)\in\Y\times\Z\times \X,
\end{equation}
which is convex in $(y,z)\in\Y\times\Z$ and concave in $x\in\X$. The Slater constraint qualification for problem (\ref{eq:1}) is said to be held  if
$$
\Big \{ (y,z)\, |\, y\in\textrm{ri}(\textrm{dom}f),\, z\in\textrm{ri}(\textrm{dom}g), \, \A^*y+\B^*z=c \Big \}\neq \varnothing, 
$$
where $\text{ri}(\cdot)$ denotes the relative interior of a convex set.
Under this constraint qualification, from \cite[Corollaries 28.2.2 and 28.3.1]{rockafellar1970convex}, we know that $( \bar{y},\bar{z})\in \textrm{ri}(\textrm{dom}f\times \textrm{dom}g)$ is an optimal solution to  problem (\ref{eq:1}) if and only if there exists a Lagrange multiplier  $\bar{x}\in \X$ such that $(\bar{y},\bar{z},\bar{x})$ is a solution to the following KKT system:
\begin{equation}\label{eq:3}
	\A{x}\in \partial f({y}),\; \; \B{x}\in \partial g({z})\; \; \textrm{and}\; \; \A^*{y}+\B^*{z}-c=0,
\end{equation}
where  $\partial f$ and $\partial g$ are the  subdifferential mappings of $f$ and $g$ in (\ref{eq:1}), respectively. Noting that the subdifferential mappings $\partial f$ and $\partial g$ are maximal monotone \cite[Theorem 12.17]{Rockafellar1998VariationalA}, then for all $y, y'\in \textrm{dom} f$,  $\xi \in \partial f(y) $ and $ {\xi }'\in \partial f({y}')$, it holds
\begin{equation*}
	\langle {\xi }'-\xi, {y}'-y\rangle\geq  \|  {y}'-y\|_{\Sigma _f}^2, 
\end{equation*}
and for all $z, z'\in \textrm{dom} g$,  $\zeta  \in \partial g(z) $ and ${\zeta  }'\in \partial g({z}')$, it holds
\begin{equation*}
	\langle  {\zeta  }'-\zeta , {z}'-z\rangle  \geq \|  {z}'-z \|_{\Sigma _g}^2,
\end{equation*}
where  $\Sigma _f:\Y\rightarrow\Y$ and $\Sigma _g:\Z\rightarrow\Z$ are self-adjoint and positive semidefinite linear operators.

Let $\bar{\Omega }:=(\bar{y},\bar{z},\bar{x})$ be the  solution set to the KKT system (\ref{eq:3}). The nonempty of $\bar{\Omega}$ can be guaranteed under the setting of a certain constraint qualification condition. 
In this paper, we only assume the existence of KKT point, and do not particularly emphasize which qualification is used.

\begin{assumption}\label{assumption1}
	The KKT system (\ref{eq:2}) has a nonempty solution set, i.e., $\bar{\Omega}\neq\varnothing$.
\end{assumption}

In order to facilitate the subsequent theoretical analysis, we let $\nu :=(x,y,\tilde{y}-y,z,\tilde{z}-z)\in \V:=\X \times \Y \times \Y \times \Z \times \Z$ for given $\tilde{y}\in\Y$ and $\tilde{z}\in\Z$, and define a KKT mapping $\hat{\R}:\V\rightarrow \V$ in the form of
\begin{equation}\label{eq:22}
	\hat{\R}(\nu):=\begin{pmatrix}
		\A^*y+\B^*z-c\\
		y-\textrm{Prox}_f(y+\A x)\\
		0\\
		z-\textrm{Prox}_g(z+\B x)\\
		0 
	\end{pmatrix},\qquad\forall\nu\in \V,
\end{equation}
where $\textrm{Prox}_f( \cdot)$ represents the proximal mapping of a closed convex proper function $f$. Define $\bar{\nu}:=(\bar{x}, \bar{y}, 0, \bar{z}, 0)$ such that $(\bar{x}, \bar{y}, \bar{z})\in\bar{\Omega}$. For simplicity, we denote a generalized optimal solutions set $\bar{\Theta}:=\left \{ (0,0) \right \}\cup \bar{\Omega }$, which means that $(\bar{x}, \bar{y}, \bar{z})\in\bar{\Omega}$ if and only if $\bar{\nu}\in\bar{\Theta}$. By optimization theory, we know that the proximal mappings $\textrm{Prox}_f( \cdot)$ and $\textrm{Prox}_g ( \cdot )$ are globally Lipschitz continuous with modulus one. Then, the mapping $\hat{\R}(\cdot)$ is continuous on $\V$ and $\hat{\R}(\bar\nu)=0$ if and only if $\bar{\nu}\in\bar{\Theta}$.

%%%%%%%%%%%%%%%%%%%%%%%%%%%%%%%%%%%%%%%%%%%
\subsection{Locally upper Lipschitzian and calmness}\label{subsec2.2}
%%%%%%%%%%%%%%%%%%%%%%%%%%%%%%%%%%%%%%%%%%%
Given a set-valued function, say $\Psi$, from $\X$ to $\Y$,  
the graph of $\Psi$ is defined as 
$\Gamma _\Psi:=\left \{ (x,y)\in \X\times\Y \ | \ y\in\Psi(x) \right \}$. For convenience, we denote $\mathbf{B}_\Y$ as an Euclidean unit ball, i.e., $\mathbf{B}_\Y:=\left \{ y\in\Y|\left \| y \right \|\leq 1 \right \}$.
\begin{definition}[\cite{robinson1981some}]
	The set-valued mapping $\Psi : \X \rightrightarrows\Y$ is said to be locally upper Lipschitzian at a point $x_0\in\X$ with modulus $\lambda$, if for some neighborhoods $\N$ of $x_0$ and for all $x\in \N$ such that
	\[\Psi (x)\subseteq \Psi(x_0)+\lambda \left \| x-x_0 \right \|\mathbf{B}_\Y,   \:\:\:\forall x\in \N.\]
\end{definition}
\noindent The set-valued mapping $\Psi$ is called piecewise polyhedral if its graph $\Gamma _\Psi$ is the union of finitely many polyhedral sets. The elementary relationship between the locally upper Lipschitzian and the piecewise polyhedral for a set-valued mapping $\Psi$ is stated as follows:
\begin{proposition}[\cite{robinson1981some}]
	Let set-valued mapping $\Psi$ be piecewise polyhedral from $\X$ into $\Y$, then there 
	exists a constant $\lambda$ such that $\Psi$ is locally upper Lipschitzian at each $x_0\in \X$ independent of the choice of $x_0$.
\end{proposition}
It was known from \cite[definition 10.20]{Rockafellar1998VariationalA} that, if a set-valued function $\Psi$ is called piecewise linear-quadratic, then dom$\Psi$ can be represented as the union of finitely many polyhedral sets, relative to each of which $\Psi$ is either given by an expression of affine or quadratic function. 
Meanwhile, $\Psi$ is piecewise linear-quadratic if and only if  the subdifferential mapping $\partial\Psi$ is piecewise polyhedral. The proof and its extensions can be found at the monograph \cite[proposition 12.30]{Rockafellar1998VariationalA}.
 
We now ready to state the definition of calmness for $\Psi : \X \rightrightarrows\Y$ at $x_0$ for $y_0$ with $(x_0,y_0)\in \Gamma _\Psi$. For more details, one can see the disquisition of Dontchev \& Rockafellar \cite{dontchev2009implicit} and Rockafellar \& Wets \cite{Rockafellar1998VariationalA}.
 \begin{definition}[\cite{dontchev2009implicit}]\label{def1}
 	A set-valued mapping $\Psi : \X \rightrightarrows\Y$ is called to be calm at $x_0$ for $y_0$ if $(x_0,y_0)\in \Gamma _\Psi$ and there exists a constant $\lambda$ along with neighborhoods $\N$ of $x_0$ and $\M$ of $y_0$ such that
 	\[\Psi(x)\cap \M\subseteq \Psi(x_0)+\lambda \left \| x-x_0 \right \|\mathbf{B}_\Y,\:\:\:\forall x\in \N.\]
 \end{definition}
\noindent As can be seen from this definition that, suppose $\Psi$ be the subdifferential mapping of a piecewise linear-quadratic function, then $\Psi$ is calm at $x^0$ for $y^0$ meeting $(x_0,y_0)\in\Gamma _\Psi$ with modulus $\lambda\geq  0$ independent of the selection of $(x_0,y_0)$. At last, a set-valued mapping $\Psi : \X \rightrightarrows\Y$ is called metrically subregular at $x_0$ for $y_0$ if $(x_0,y_0)\in \Gamma _\Psi$ and there exists a constant $\iota\geq0$ along with neighborhoods $\N$ of $x_0$ and $\M$ of $y_0$ such that
$$
\textrm{dist} (x,\Psi ^{-1}(y_0))\leq \iota \, \textrm{dist}(y_0,\Psi(x)\cap \M),\:\:\:\forall x\in \N,
$$
which is a.k.s.  error bound condition.
To end this section,  we list the following  result  to reveal the equivalence of metric subregularity of a set-valued mapping with calmness of its inverse. For its proof, one can refer to \cite[Theorem 3H.3]{dontchev2009implicit}.
\begin{proposition}[\cite{dontchev2009implicit}]
	For a set-valued mapping $\Psi : \X \rightrightarrows\Y$, let $(x_0,y_0)\in \Gamma _\Psi$. Then $\Psi$ is metrically subregular at $x_0$ for $y_0$ with a constant $\lambda$ if and only if its inverse $\Psi^{-1}: \Y \rightrightarrows\X$ is calm at $y_0$ for $x_0$ with the same constant $\lambda$.
\end{proposition}

%%%%%%%%%%%%%%%%%%%%%%%%%%%%%%%%%%%%%%%%%%%
\section{Linear convergence rate}\label{sec3}
%%%%%%%%%%%%%%%%%%%%%%%%%%%%%%%%%%%%%%%%%%%
In this section, we present a general convergence rate analysis on algorithm p-GADMM. 
It should be noted that the steps of p-GADMM has been fully stated in 
\cite{xiao2018generalized}, but for the convenience of the subsequent analysis, we adjust the updating order and the upper script here.

\begin{algorithm}
	\renewcommand{\thealgorithm}{1}
	\caption{p-GADMM}\label{alg:2}
	\begin{itemize}
		\item[Step 0.] Let $\sigma \in(0,+\infty)$ and $\rho \in (0,2)$ be given parameters. Let $\S$ and $\T$ be  self-adjoint positive definite linear operators defined on $\Y$ and $\Z$, respectively. Choose $(\tilde{x}^0,\tilde{y}^0,\tilde{z}^1)\in\X\times\textrm{dom}f\times \textrm{dom}g$.		
		\item[Step 1.] Compute
		\begin{equation*}
			\begin{cases}
				y^0:=\underset{y\in\Y}{\arg\min}\L_{\sigma }(y,\tilde{z}^1;\tilde{x}^0)+\frac{1}{2}\left \| y-\tilde{y}^0 \right \|^2_\S,\\ 
				x^0:=\tilde{x}^0-\sigma (A^*y^0+B^*\tilde{z}^1-c).
			\end{cases}
		\end{equation*}
		\item[Step 2.]  For $k=1,2,3,\ldots$, do the following steps iteratively:
		\begin{subequations}
			\begin{numcases}{} 
				z^k:=\underset{z\in\Z}{\arg\min}\L_\sigma (y^{k-1},z;x^{k-1})+\frac{1}{2}\left \| z-\tilde{z}^k \right \|^2_\T,\label{eq:a}\\
				\tilde{y}^k:=\tilde{y}^{k-1}+\rho (y^{k-1}-\tilde{y}^{k-1}),\label{eq:b}\\
				\tilde{x}^k:=\tilde{x}^{k-1}+\rho (x^{k-1}-\tilde{x}^{k-1}),\label{eq:c}\\
				\tilde{z}^{k+1}:=\tilde{z}^{k}+\rho (z^{k}-\tilde{z}^{k}),\label{eq:d}\\
				y^k:=\underset{y\in\Y}{\arg\min}\L_\sigma (y,\tilde{z}^{k+1};\tilde{x}^{k})+\frac{1}{2}\left \| y-\tilde{y}^k \right \|^2_\S,\label{eq:e}\\ 
				x^k:=\tilde{x}^k-\sigma (\A^*y^k+\B^*\tilde{z}^{k+1}-c). \label{eq:f}
			\end{numcases}
		\end{subequations}
		\item[Step 3.] If a termination criterion is not met, set $k:=k+1$ and go to Step 2.
		%\vspace{-.3cm}
	\end{itemize}
\end{algorithm}

From \cite[Lemma 5.2, Theorem 5.1]{xiao2018generalized}, it is a trivial task to get the following result which provides some useful highlights for the further convergence rate analysis.
\begin{theorem}\label{th:1}
	Suppose that the KKT system (\ref{eq:3}) is nonempty. Let the sequence $\{(x^k, y^k, z^k; \tilde{x}^k,\tilde{y}^k,\tilde{z}^k) \}$ be generated by Algorithm \ref{alg:2}. Then the following results hold:
	\begin{itemize}
		\item [(i)] For any $k\geq0$,
		\begin{equation}\label{eq:4}
			\begin{split}
			&(\sigma \rho )^{-1}\left \| x_e^k+\sigma (1-\rho ) \A^*y_e^k\right \|^2+\rho ^{-1}\left \| \tilde{y}_e^{k+1} \right \|_\S^2+\rho ^{-1}\left \| \tilde{z}_e^{k+1} \right \|_\T^2\\[2mm]
			&+(2-\rho )\left \| y^k-\tilde{y}^k \right \|_S^2+\left ( 2-\rho  \right )\sigma \left \| \A^*y_e^k \right \|^2\\[2mm]
			\geq &(\sigma \rho )^{-1}\left \| x_e^{k+1}+\sigma (1-\rho ) \A^*y_e^{k+1}\right \|^2+\rho ^{-1}\left \| \tilde{y}_e^{k+2} \right \|_S^2+\rho ^{-1}\left \| \tilde{z}_e^{k+2} \right \|_\T^2\\[2mm]
			&+(2-\rho )\left \| y^{k+1}-\tilde{y}^{k+1} \right \|_\S^2+\left ( 2-\rho  \right )\sigma \left \| \A^*y_e^{k+1} \right \|^2\\[2mm]
			&+2\left \| y_e^{k+1} \right \|_{\Sigma _f}^2+2\left \| z_e^{k+1} \right \|_{\Sigma _g}^2+(2-\rho )\sigma \left \| \A^*y_e^{k+1} +\B^*z_e^{k+1}\right \|^2\\[2mm]
			&+(2-\rho )\left \| \tilde{y}^{k+1}-y^{k+1} \right \|_S^2+(2-\rho )\left \| \tilde{z} ^{k+1}-z^{k+1}\right \|_\T^2+\sigma \rho ^{-1}(2-\rho )^2\left \| \A^*(y^{k+1}-y^k) \right \|^2,
		    \end{split}
		\end{equation}
	where $x_e=x-\bar{x}$, $y_e=y-\bar{y}$, and $z_e=z-\bar{z}$.
	\item[(ii)] Assume that both $\S$ and $\T$ be chosen such that $\Sigma _f+\S+\sigma\A\A^*\succ 0$ and $\Sigma _g+\T+\sigma\B\B^*\succ 0$, then the sequence $\left \{ (x^k, y^k, \tilde{y}^k-y^k, z^k, \tilde{z}^k-z^k)\right \}$ is automatically well-defined, and it converges to $(\bar{x}, \bar{y}, 0, \bar{z}, 0)\in\bar{\Theta}$.
		\end{itemize}
\end{theorem}

Theorem \ref{th:1} presents a global convergence result for p-GADMM under fairly general and mild conditions.  Evidently, one can choose positive semidefinite (and even indefinite) linear operators $\S$ and $\T$ to ensure $\Sigma _f+\S+\sigma\A\A^*\succ 0$ and $\Sigma _g+\T+\sigma\B\B^*\succ 0$. But, due to the existences of some coupling terms in (\ref{eq:4}), we 
must restrict $\S$ and $\T$ to be positive definite. In this case, the conditions $\Sigma _f+\S+\sigma\A\A^*\succ 0$ and $\Sigma _g+\T+\sigma\B\B^*\succ 0$ hold automatically.

We now present some notations to facilitate the later theoretical analysis.  For any self-adjoint linear operator $\G: \X\rightarrow\X$, we use the symbol $\lambda_{\max}(\G)$ to denote its largest eigen-value.
Denote
\begin{equation}\label{eq:8}
	\Upsilon(\nu^{k+1}):=\kappa\left(\left \| \tilde{y}^{k+1}-y^{k+1} \right \| _\S^2+\left \| \tilde{z}^{k+1}-z^{k+1} \right \|_\T^2+\left \| \A^*y_e^{k+1}+\B^*z_e^{k+1} \right \|^2+\left \| \A^*(y^{k+1}-y^k) \right \|^2\right),
\end{equation}
where
\begin{equation}\label{kapa}
	\kappa :=\max\bigg \{ \|\S \|, 3\|\T \|,3(2-\rho )^2\sigma ^2\lambda _{\max}(\B^*\B), 3(1-\rho )^2\sigma ^2\lambda _{\max}(\B^*\B)+1\bigg \}.
\end{equation}
Moreover, denote $\Xi:\V\rightarrow\V$ in the form of
\begin{equation*}
	\Xi:=\left(
	\begin{smallmatrix}
	(\sigma \rho)^{-1}\I & (1-\rho)\rho^{-1}\A^* &  0& 0 & 0\\ 
	(1-\rho)\rho^{-1}\A &\rho^{-1}\sigma\A\A^*+\rho^{-1}S+2\Sigma _f  & (1-\rho)\rho^{-1}\S & 0 &0 \\ 
	0&  (1-\rho)\rho^{-1}\S & \rho^{-1}\S & 0 & 0\\ 
	0 &  0&  0& \rho^{-1}\T+2\Sigma _g & (1-\rho)\rho^{-1}\T\\ 
	0 & 0 & 0 &   (1-\rho)\rho^{-1}\T&  (1-\rho)^2\rho^{-1}\T
\end{smallmatrix}
\right)
+\frac{1}{2}(2-\rho)\sigma\vartheta\vartheta^*,
\end{equation*}
where $\I$ is an identity operator, and $\vartheta :\X\rightarrow \V$ is a linear operator such that its adjoint $\vartheta^*$ satisfies $\vartheta^*(\nu)= \A^*y+\B^*z$ for any $\nu\in\V$. 
In the subsequent analysis, we use the operator $\Xi$ to measure the weighted distance from the current point to the generalized optimal solutions set $\bar{\Theta}$. Obviously, if $\rho\in(0,2)$ and $\S\succ 0$, $\T\succ 0$, then $\Xi$ must be positive definite, i.e.,
$$
\{\S\succ 0 \ \& \ \T\succ 0\}\; \Leftrightarrow \; \Xi \succ 0,
$$
which plays a key rule in the linear convergence rate result. 
 Additionally, to conduct the rate of the decrease of $\left \| \nu ^k-\bar{\nu} \right \|^2$, we take an interest in deducing an upper bound for $\hat{R}(\cdot)$ computed at the sequence generated by the p-GADMM in the subsequent developments.

\begin{lemma}\label{lemma1}
	Let $\left \{ \nu^k \right \}$ be the infinite sequence generated by p-GADMM. Then for any $k\geq1$, we have 
	\begin{equation}\label{ineq34}
		\Upsilon(\nu^{k+1})\geq \| \hat{\R}(\nu ^{k+1})  \|^2.
	\end{equation}
	\end{lemma}
\begin{proof}
	From (\ref{eq:c}) and (\ref{eq:f}), it is easy to see that
	\begin{equation*}
		\tilde{x}^{k+1}=x^k+(1-\rho )(\tilde{x}^k-x^k)=x^k+\sigma (1-\rho)(\A^*y_e^k+\B^*\tilde{z}_e^{k+1}).
	\end{equation*}
Then, substituting this equality into (\ref{eq:f}), we get
\begin{equation}\label{eq:7}
	\begin{split}
		x^{k+1}&=\tilde{x}^{k+1}-\sigma (\A^*y_e^{k+1}+\B^*\tilde{z}_e^{k+2})\\[2mm]
		&=x^k-\sigma \rho (\A^*y_e^{k+1}+\B^*z_e^{k+1})+\sigma (1-\rho)\A^*(y^{k}-y^{k+1}).
	\end{split}
\end{equation}
	By the first order optimality condition of (\ref{eq:a}), we have 
	\begin{equation*}
		\begin{split}
			0&\in\partial g(z^{k+1})-\B x^k+\sigma \B(\A^*y_e^k+\B^*z_e^{k+1})+\T(z^{k+1}-\tilde{z}^{k+1})\\[2mm]
			&=\partial g(z^{k+1})-\B\left (  x^k-\sigma (\A^*y_e^{k+1}+\B^*z_e^{k+1})+\sigma \A^*(y^{k+1}-y^k)\right )+\T(z^{k+1}-\tilde{z}^{k+1}),	
		\end{split}
	\end{equation*}
		which leads to an equivalent expression for $z^{k+1}$, that is,
		\begin{equation}\label{eq:5}
			z^{k+1}=\textrm{Prox}_g(z^{k+1}+\B\left (  x^k-\sigma (\A^*y_e^{k+1}+\B^*z_e^{k+1})+\sigma \A^*(y^{k+1}-y^k)\right )-\T(z^{k+1}-\tilde{z}^{k+1})).
		\end{equation}
	It follows from (\ref{eq:e}) we can easily get
	\begin{equation*}
		0\in \partial f(y^{k+1})-\A\tilde{x}^{k+1}+\sigma \A(\A^*y_e^{k+1}+\B^*\tilde{z}_e^{k+2})+\S(y^{k+1}-\tilde{y}^{k+1}),
	\end{equation*}
which, from (\ref{eq:f}), is equivalent to
\begin{equation*}
	0\in \partial f(y^{k+1})-\A x^{k+1}+\S(y^{k+1}-\tilde{y}^{k+1}).
\end{equation*}
Thus, we obtain that
\begin{equation}\label{eq:6}
	y^{k+1}=\textrm{Prox}_f\Big(y^{k+1}+\A x^{k+1}-\S(y^{k+1}-\tilde{y}^{k+1})\Big).
\end{equation}
Secondly, associating with the equations (\ref{eq:7}), (\ref{eq:5}) and (\ref{eq:6}) and using the Lipschitz continuity of Moreau-Yosida proximal mapping, we get from the definition of $\hat{\R}(\cdot)$ in (\ref{eq:22}) that
\begin{equation*}
	\begin{split}
		\left \| \hat{\R}(\nu ^{k+1}) \right \|^2\leq &\left \| \A^*y_e^{k+1}+\B^*z_e^{k+1} \right \|^2+\left \| \S(y^{k+1}-\tilde{y}^{k+1}) \right \|^2\\[2mm]
		&+\left \| \B(x^k-x^{k+1})-\sigma \B(\A^*y_e^{k+1}+\B^*z_e^{k+1})+\sigma \B\A^*(y^{k+1}-y^k)-\T(z^{k+1}-\tilde{z}^{k+1}) \right \|^2\\[2mm]
		=&\left \| \A^*y_e^{k+1}+\B^*z_e^{k+1} \right \|^2+\left \| \S(y^{k+1}-\tilde{y}^{k+1}) \right \|^2\\[2mm]
		&+\left \| (\rho -1)\sigma \B(\A^*y_e^{k+1}+\B^*z_e^{k+1})+(2-\rho )\sigma \B\A^*(y^{k+1}-y^k)-\T(z^{k+1}-\tilde{z}^{k+1}) \right \|^2\\[2mm]
		\leq&\left \| \A^*y_e^{k+1}+\B^*z_e^{k+1} \right \|^2+\left \| \S \right \|\left \| y^{k+1}-\tilde{y}^{k+1} \right \|_\S^2+3\left \| \T \right \|\left \| z^{k+1}-\tilde{z}^{k+1} \right \|_\T^2\\[2mm]
		&+3(1-\rho )^2\sigma ^2\lambda_{\max}(\B^*\B)\left \| \A^*y_e^{k+1}+\B^*z_e^{k+1} \right \|^2+3(2-\rho )^2\sigma ^2\lambda_{\max}(\B^*\B)\left \| \A^*(y^{k+1}-y^k) \right \|^2\\[2mm]
		\leq&\kappa(\left \| \A^*y_e^{k+1}+\B^*z_e^{k+1} \right \|^2+\left \| y^{k+1}-\tilde{y}^{k+1} \right \|_\S^2+\left \| z^{k+1}-\tilde{z}^{k+1} \right \|_\T^2+\left \| \A^*(y^{k+1}-y^k) \right \|^2).
	\end{split}
\end{equation*}
Using the definition of $\Upsilon(\cdot)$ in (\ref{eq:8}), we get the inequality (\ref{ineq34}).
\end{proof}

To get the local linear convergence rate of p-GADMM \ref{alg:2}, we need another assumption to control the distance from an iterate $\nu$ to the KKT solution set $\bar{\Theta}$.
\begin{assumption}\label{ass2}
	The inverse of the  mapping $\hat{\R}(\cdot)$ defined in (\ref{eq:22}) is calm at the $0\in\V$ for $\bar{\nu}$ with modulus $\lambda>0$ if there exists a constant $\varepsilon>0$ such that 
    $$
	\text{dist}(\nu, \bar{ \Theta })\leq \lambda \| \hat{\R}(\nu) \|, \; \; \; \forall \nu\in \Big\{ \nu\in\V| \| \nu-\bar{\nu}  \|\leq \varepsilon  \Big\}.
	$$
\end{assumption}
In light of above analysis, we are ready to establish the local linear convergence rate result of algorithm p-GADMM.
\begin{theorem}\label{th2}
	Suppose that Assumptions \ref{assumption1} and  \ref{ass2} hold. Besides, assume that $\S$ and $\T$ are positive definite. Let the sequence $\{(x^k, y^k, z^k; \tilde{x}^k,\tilde{y}^k,\tilde{z}^k) \}$ be generated by Algorithm p-GADMM. 	
	Then $\{ \nu^k :=(x^k, y^k, \tilde{y}^k-y^k, z^k, \tilde{z}^k-z^k) \}$  converges to $\bar{\nu} =(\bar{x}, \bar{y}, 0, \bar{z}, 0)\in\bar{\Theta}$, and there exists a threshold $\bar{\kappa}\geq1$ such that for all $k\geq \bar{\kappa}$, it holds that
	\begin{equation}\label{eq:14}
		\textrm{dist}_{\Xi}^2\left (\nu^{k+1}  ,\bar{\Theta }\right )\leq \alpha\textrm{dist}_{\Xi}^2\left (\nu^{k}  ,\bar{\Theta }\right ),
	\end{equation}
where
$$
\alpha:=(1+\beta)^{-1} \:\:\:\textrm{and}\:\:\:\beta:=(2-\rho)\min \left \{ 1,\frac{1}{2} \sigma ,\sigma \rho ^{-1}(2-\rho)\right \}\left(\lambda ^{2}\kappa\lambda_{\max}(\Xi) \right)^{-1},
$$
where $\kappa$ is defined in (\ref{kapa}).
Moreover, there exists a positive number $\zeta\in[\alpha,1)$ such that for all $k\geq\bar{\kappa}$
\begin{equation}\label{eq:15}
	\textrm{dist}_{\Xi}^2\left (\nu^{k+1}  ,\bar{\Theta }\right )\leq \zeta\textrm{dist}_{\Xi}^2\left (\nu^{k}  ,\bar{\Theta }\right ).
\end{equation}
\end{theorem}
\begin{proof}
From the part  $(i)$ of Theorem \ref{th:1}, it holds that
	\begin{equation*}
			\begin{split}
				&(\sigma \rho )^{-1}\left \| x_e^k+\sigma (1-\rho ) \A^*y_e^k\right \|^2+\rho ^{-1}\left \| \tilde{y}_e^{k+1} \right \|_\S^2+\rho ^{-1}\left \| \tilde{z_e}^{k+1} \right \|_\T^2+\left ( 2-\rho  \right )\sigma \left \| \A^*y_e^k \right \|^2\\[2mm]
				&+(2-\rho )\left \| y^k-\tilde{y}^k \right \|_\S^2+\frac{1}{2}(2-\rho )\sigma \left \| \vartheta^*(\nu_e^k)\right \|^2+2\left \| y_e^{k} \right \|_{\Sigma _f}^2+2\left \| z_e^{k} \right \|_{\Sigma _g}^2\\[2mm]
				\geq &(\sigma \rho )^{-1}\left \| x_e^{k+1}+\sigma (1-\rho ) \A^*y_e^{k+1}\right \|^2+\rho ^{-1}\left \| \tilde{y}_e^{k+2} \right \|_\S^2+\rho ^{-1}\left \| \tilde{z_e}^{k+2} \right \|_\T^2+(2-\rho )\left \| y^{k+1}-\tilde{y}^{k+1} \right \|_\S^2\\[2mm]
				&+\left ( 2-\rho  \right )\sigma \left \| \A^*y_e^{k+1} \right \|^2+2\left \| y_e^{k+1} \right \|_{\Sigma _f}^2+2\left \| z_e^{k+1} \right \|_{\Sigma _g}^2+\frac{1}{2}(2-\rho )\sigma \left \| \vartheta^*(\nu_e^{k+1})\right \|^2\\[2mm]
				&+(2-\rho )\left \| \tilde{y}^{k+1}-y^{k+1} \right \|_\S^2+(2-\rho )\left \| \tilde{z} ^{k+1}-z^{k+1}\right \|_\T^2\\[2mm]
				&+\sigma \rho ^{-1}(2-\rho )^2\left \| \A^*(y^{k+1}-y^k) \right \|^2+\frac{1}{2}(2-\rho )\sigma \left \| \A^*y_e^{k+1}+\B^*z_e^{k+1}\right \|^2,
			\end{split}
		\end{equation*}
	which implies that
	\begin{equation}\label{eq:13}
		\begin{split}
			\left \| \nu_e^k \right \|_\Xi ^2\geq &\left \| \nu_e^{k+1} \right \|_\Xi ^2+(2-\rho )\left (\left \| \tilde{y}^{k+1}-y^{k+1} \right \|_\S^2+ \left \| \tilde{z} ^{k+1}-z^{k+1}\right \|_\T^2\right.\\[2mm]
			&\left.+\sigma \rho ^{-1}(2-\rho )\left \| \A^*(y^{k+1}-y^k) \right \|^2+\frac{1}{2}\sigma \left \| \A^*y_e^{k+1}+\B^*z_e^{k+1}\right \|^2 \right ).
		\end{split}
	\end{equation}
Because $\bar{ \Theta }$ is a nonempty closed convex set, we can immediately get (\ref{eq:13}) to the following required result
\begin{equation}\label{eq:9}
	\begin{split}
		\text{dist}^2_\Xi(\nu^{k},\bar{\Theta})\geq &\:\text{dist}^2_\Xi(\nu^{k+1},\bar{\Theta})+(2-\rho )\left (\left \| \tilde{y}^{k+1}-y^{k+1} \right \|_\S^2+ \left \| \tilde{z} ^{k+1}-z^{k+1}\right \|_\T^2\right.\\
		&\left.+\sigma \rho ^{-1}(2-\rho )\left \| \A^*(y^{k+1}-y^k) \right \|^2+\frac{1}{2}\sigma \left \| \A^*y_e^{k+1}+\B^*z_e^{k+1}\right \|^2 \right ).
	\end{split}
\end{equation}
Observing the structure of right hand side of (\ref{eq:9}) and  the definition of $\Upsilon(\nu^{k+1})$ in (\ref{eq:8}), we know for all $k\geq1$ and $\rho\in(0,2)$ that 
\begin{equation}\label{eq:10}
	\begin{split}
			&\kappa\left(\left \| \tilde{y}^{k+1}-y^{k+1} \right \|_\S^2+ \left \| \tilde{z} ^{k+1}-z^{k+1}\right \|_\T^2+\sigma \rho ^{-1}(2-\rho )\left \| \A^*(y^{k+1}-y^k) \right \|^2+\frac{1}{2}\sigma \left \| \A^*y_e^{k+1}+\B^*z_e^{k+1}\right \|^2\right)\\
			\geq&\min \left \{ 1,\frac{1}{2} \sigma,\sigma \rho ^{-1}(2-\rho)\right \}\Upsilon(\nu^{k+1}).
			\end{split}
\end{equation}
According to part $(ii)$ of Theorem \ref{th:1}, we know that the sequence $\left \{ (x^k, y^k, \tilde{y}^k-y^k, z^k, \tilde{z}^k-z^k)\right \}$ converges to $\bar{\nu}^k =(\bar{x}, \bar{y}, 0, \bar{z}, 0)$, which means that there exists $\bar{\kappa}\geq1$ and $\varepsilon>0$ such that for any $k\geq\bar{\kappa}$, it holds that
\[\left \| \nu^{k+1}-\bar{\nu} \right \|\leq \varepsilon. \]
Subsequently, from Assumption \ref{ass2} and Lemma \ref{lemma1}, it gets for all $k\geq \bar{\kappa}$ that
\begin{equation}\label{eq:11}
	\text{dist}^2(\nu^{k+1},\bar{\Theta})\leq \lambda ^2\left \| \hat{\R}(\nu^{k+1}) \right \|^2\leq \lambda ^2\Upsilon(\nu^{k+1}).
\end{equation}
Combining (\ref{eq:10}) with (\ref{eq:11}) and using the fact $\rho\in(0,2)$, it yields that
\begin{equation}\label{eq:12}
	\begin{split}
		&(2-\rho)\left(\left \| \tilde{y}^{k+1}-y^{k+1} \right \|_\S^2+ \left \| \tilde{z} ^{k+1}-z^{k+1}\right \|_\T^2\right.\\
		&\left.+\sigma \rho ^{-1}(2-\rho )\left \| \A^*(y^{k+1}-y^k) \right \|^2+\frac{1}{2}\sigma \left \| \A^*y_e^{k+1}+\B^*z_e^{k+1}\right \|^2\right)\\
		\geq&(2-\rho)\min \left \{ 1,\frac{1}{2} \sigma ,\sigma \rho ^{-1}(2-\rho)\right \}\lambda ^{-2}\kappa ^{-1}\text{dist}^2(\nu^{k+1},\bar{\Theta})\\
		\geq&\beta\;\text{dist}^2_\Xi(\nu^{k+1},\bar{\Theta}).
	\end{split}
\end{equation}
From (\ref{eq:12}), we can derive the assertion (\ref{eq:14}).
Recalling that $\alpha<1$, this readily ensures the fulfillment of linear convergence rate (\ref{eq:15}).
\end{proof}

From Theorem \ref{th2},  we know that the conclusion of local linear convergence rate for Algorithm p-GADMM  relies on the calmness property. 
As discussed in section \ref{subsec2.2} that, for any set-valued mapping $\Psi$, the property of calmness at a certain point holds automatically if $\Psi$ is piecewise polyhedral, and particularly, $\Psi$ is a subdifferential mapping of a convex piecewise linear-quadratic function.
Based on the fact that $\Psi^{-1}$ is piecewise polyhedral if and only if $\Psi$ itself is piecewise polyhedral, 
we can get that when piecewise polyhedral condition is imposed on the  mapping $\hat{\R}(\cdot)$, then the global linear convergence rate of Algorithm p-GADMM can be derived. 

\begin{corollary}
	Let the sequence $\{(x^k, y^k, z^k; \tilde{x}^k,\tilde{y}^k,\tilde{z}^k) \}$ be generated by Algorithm p-GADMM. 		
	Let $\nu^k :=(x^k, y^k, \tilde{y}^k-y^k, z^k, \tilde{z}^k-z^k)$ and  $\rho\in(0,2)$, and let $\bar{\nu} =(\bar{x}, \bar{y}, 0, \bar{z}, 0)\in\bar{\Theta}$ be any limiting point of $\left \{ \nu^{k} \right \}$. Suppose that the solution set to the KKT system (\ref{eq:3}) is nonempty and that $\S$ and $\T$ are positive definite. Besides, suppose that the mapping $\hat{\R}(\cdot)$ is piecewise polyhedral. Then, there exists a constant $\hat{\lambda }>0$ such that for all $k\geq 1$ ,
	\begin{equation}\label{eq:37}
		\textrm{dist}\left ( \nu^k,\bar{\Theta } \right )\leq \hat{\lambda } \| \R\left ( \nu^k \right ) \|,
	\end{equation}
	and
	\begin{equation}\label{eq:38}
		\textrm{dist}_{\Xi}^2\left (\nu^{k+1}  ,\bar{\Theta }\right )\leq  \hat{\alpha}\textrm{dist}_{\Xi}^2\left (\nu^{k}  ,\bar{\Theta }\right ),
	\end{equation}
	where 
	$$
	\hat{\alpha}:=(1+\hat{\beta})^{-1} \:\:\:\textrm{and}\:\:\:\hat{\beta}:=(2-\rho)\min \left \{ 1,\frac{1}{2} \sigma ,\sigma \rho ^{-1}(2-\rho)\right \}\left(\hat{\lambda} ^{2}{\kappa}\lambda_{\max}(\Xi) \right)^{-1}.
	$$
\end{corollary}
\begin{proof}
	On the one hand, from the nonempty set $\bar{\Theta}$ and the piecewise polyhedral condition, there exist fixed $\lambda>0$ and $\delta>0$ such that 
	$$
	\textrm{dist}(\nu^k,\bar{\Theta})\leq\lambda \| \hat{\R}(\nu^k)  \|
	$$
	if $ \| \hat{\R}(\nu^k)\| \leq \delta$ .
	On the other hand, following the proof of Theorem \ref{th2}, for all $k\geq1$, we know that  $ \{\nu^k:= (x^k, y^k, \tilde{y}^k-y^k, z^k, \tilde{z}^k-z^k)  \}$  converges to $\bar{\nu} =(\bar{x}, \bar{y}, 0, \bar{z}, 0)$ while $\left \| \nu^k-\bar{\nu} \right \|\leq \varepsilon$ with constant $\varepsilon>0$. For the $\nu^k$ satisfying $\| \hat{\R}(\nu^k) \|>\delta$, we get 
	\[\textrm{dist}(\nu^k, \bar{\Theta })\leq\left \| \nu^k-\bar{\nu} \right \|\leq\varepsilon<\varepsilon\delta^{-1} \| \hat{\R}(\nu^k) \|.\]
	We readily obtain that there exists a positive number $\hat{\lambda}:=\max\left \{ \lambda, \varepsilon\delta^{-1} \right \}$ such that (\ref{eq:37}) holds. 
	Employing a similar proof of Theorem \ref{th2} yields the inequality (\ref{eq:38}), thereby the global linear convergence rate holds.
\end{proof}
At the end of this section, we notice that the assumption on $\hat{\R}$  implies that the condition $\textrm{dist}\left ( \nu^k,\bar{\Theta } \right )\leq \hat{\lambda } \| \R\left ( \nu^k \right ) \|$ hold automatically. Therefore, the global linear convergence rate can be achieved.

\section{Conclusion}\label{sec4}

We know that the method of p-GADMM proposed by Xiao, Chen \& Li \cite{xiao2018generalized}  is highly efficient for convex composite programming problems. The global convergence of p-GADMM is known, but its convergence rate is worthy of exploring.
This paper was devoting to provide a theoretical analysis and proved that p-GADMM has Q-linear convergence rate under the assumption that the KKP mapping is calm. 
This conclusion is consistent with the theoretical result of Han et al. \cite{Han2018LinearRC} to  the semi-proximal ADMM of Fazel et al. \cite{Fazel2013HankelMR}.
Nevertheless, it is still worth emphasizing that Theorem \ref{th2} requires  $\S$ and $\T$ being positive definite, which is slightly stronger than the one in \cite{Han2018LinearRC}. Despite this, we believe that this condition will not affect the contribution of this paper because it actually common in vast majority of practical implementations.

\section{Acknowledgments}
The work of Y. Xiao is supported by the National Natural Science Foundation of China (Grant
No. 11971149).

\bibliography{references3}  %参考文献库的名字Ref
\bibliographystyle{plain}
\end{document}